\documentclass[reqno, 12pt]{amsart}
\usepackage{amsmath}
\usepackage{amssymb}
\usepackage{amsthm}
\usepackage{mathtools}
\usepackage{tikz-cd}
\usepackage{enumerate}
\usepackage[mathscr]{eucal}
\usepackage{graphicx}

\newtheorem{theorem}{Theorem}[section]
\newtheorem{lemma}[theorem]{Lemma}
\newtheorem{proposition}[theorem]{Proposition}
\newtheorem{corollary}[theorem]{Corollary}
\newtheorem{definition}[theorem]{Definition}
\theoremstyle{definition}
\newtheorem{example}[theorem]{Example}
\newtheorem{remark}[theorem]{Remark}

\newcommand{\trinorm}{\vert\hspace{-0.5mm}\vert\hspace{-0.5mm}\vert}
\begin{document}
	\title[]{The finest locally convex topology of an extended locally convex space}
	\author{Akshay Kumar \and Varun Jindal}
	
	
	\address{Akshay Kumar: Department of Mathematics, Malaviya National Institute of Technology Jaipur, Jaipur-302017, Rajasthan, India}
	\email{akshayjkm01@gmail.com}
	
	\address{Varun Jindal: Department of Mathematics, Malaviya National Institute of Technology Jaipur, Jaipur-302017, Rajasthan, India}
	\email{vjindal.maths@mnit.ac.in}
	

	\begin{abstract}
	Salas and  Garc{\'\i}a introduced the concept of an extended locally convex space in \cite{esaetvs} which extends the idea of an extended normed space (introduced by Beer in \cite{nwiv}). This article gives an attractive formulation of the finest locally convex topology of an extended locally convex space and provides a systematic study of the resulting locally convex space. As an application, we characterize the coincidence of the finest locally convex topologies corresponding to the topologies of uniform and strong uniform convergences on a bornology for the function space $C(X)$.  
	\end{abstract}

	

\maketitle

	\section{Introduction}
	For a vector space $X$ over a field $\mathbb{K}$ ($\mathbb{R}$ or $\mathbb{C}$), an extended norm on $X$ is a function satisfying all the properties of a norm and, in addition, can also attain infinite value. A vector space $X$ together with an extended norm is called an extended normed linear space. These spaces were first studied by Beer in \cite{nwiv} and further developed by Beer and Vanderwerff (\cite{socsiens, spoens}).  Of course, an extended norm on a vector space $X$ induces an extended metric on $X$ as follows $d(x,y) = \parallel x-y\parallel$. We refer readers to \cite{tsoervms, lmrvcf} for more details on extended metric spaces.   
	 
	 
Motivated from the work of Beer and Vanderwerff, in \cite{esaetvs}, Salas and Garc{\'\i}a introduced the concept of an extended locally convex space. They proved that every extended locally convex topology on a vector space $X$ is induced by a collection of extended seminorms on $X$. An extended seminorm on a vector space $X$ is a function that satisfies all the properties of a seminorm and can also attain infinite value. So every extended normed space is an extended locally convex space. An extended locally convex space fails to be topological vector space as the scalar multiplication need not be jointly continuous. Therefore the techniques of classical functional analysis may not apply to these new structures.  
	 
	 
In \cite{spoens}, Beer and Vanderwerff constructed the finest locally convex topology (denoted by $\tau_{P_X}$) for an extended normed linear space $(X,\parallel\cdot\parallel)$ which is coarser than the extended norm topology. It is shown that the locally convex space $(X,\tau_{P_X})$ has the same dual as that of $(X,\parallel\cdot\parallel)$. This could be an interesting approach to apply the classical locally convex theory to study these extended structures.    
	 
In this paper, we construct the finest  locally convex topology (denoted by $\tau_F$) for an extended locally convex space $(X,\tau)$ which is coarser than the extended locally convex topology $\tau$. In this setting also, we show that the dual of the finest locally convex space $(X,\tau_F)$ is the same as that of $(X,\tau)$. As an application to function spaces, we show that for a bornology $\mathcal{B}$ on a metric space $(X,d)$, the finest locally convex topologies for the spaces $(C(X),\tau_{\mathcal{B}})$ and $(C(X),\tau_{\mathcal{B}}^{s})$ coincide if and only if $\mathcal{B}$ is shielded from closed sets, where   $\tau_{\mathcal{B}}$, $\tau_{\mathcal{B}}^{s}$ denote respectively, the classical topology of uniform convergence on $\mathcal{B}$ and the topology of strong uniform convergence on $\mathcal{B}$ (introduced by Beer and Levi in \cite{Suc}).  This result complements Theorem 4.1 of \cite{Ucucas}.

In order to provide a systematic study of the topology $\tau_F$ for an extended locally convex space $(X,\tau)$, we further investigate:

\begin{itemize}
	\item  separation of a point from a closed convex set in $(X,\tau)$;
	\item  the finest locally convex topology for a subspace, quotient space, and product of extended locally convex spaces;
	\item extended normability of $(X,\tau)$;
	\item metrizability and normability of $(X,\tau_F)$.   
\end{itemize} 

%

\section{Preliminaries}

Throughout this paper, the underlying field of every vector space is either $\mathbb{R}$ or $\mathbb{C}$. For a vector space $X$, an \textit{extended norm} $\parallel\cdot\parallel:X\rightarrow [0,\infty]$ is a function  which satisfies  
\begin{itemize}
	\item[(1)] $\parallel x\parallel=0$ if and only if $x=0_X$;
	\item[(2)] $\parallel \alpha x\parallel = |\alpha |\parallel x\parallel$ for each scalar  $\alpha$ and $x\in X$;
	\item[(3)] $\parallel x+y\parallel~\leq~ \parallel x\parallel+\parallel y\parallel$ for each $x,y\in X$.
\end{itemize}
A vector space $X$ together with an extended norm $\parallel \cdot\parallel$ is called an \textit{extended normed linear space} (enls, for short), and is denoted by $(X,\parallel \cdot\parallel)$. The \textit{finite subspace} $X_{fin}$ for an enls $(X,\parallel\cdot\parallel)$ is defined as $X_{fin}=\{x\in  X : \parallel x\parallel<\infty\}$. The subspace $X_{fin}$ of an enls $X$ is the smallest open subspace of $X$ (see Corollary 3.9 of \cite{nwiv}). An enls $(X,\parallel\cdot\parallel)$ is called an extended Banach space if every Cauchy sequence in $(X,\parallel\cdot\parallel)$ converges to a point of $X$. In Proposition 3.11 of \cite{nwiv}, it is shown that an enls $(X,\parallel\cdot\parallel)$ is an extended Banach space if and only if $(X_{fin}, \parallel\cdot\parallel)$ is a Banach space. For further details on enls, we refer to \cite{nwiv} and \cite{spoens}.  

A natural generalization of an extended norm is  an extended seminorm.	
\begin{definition}{\normalfont (\cite{esaetvs})} \normalfont	A function $\rho:X\rightarrow [0, \infty]$ on a vector space $X$ is said to be an \textit{extended seminorm} if it satisfies
	\begin{itemize}
		\item[(1)] $\rho(\alpha x)=|\alpha|\rho(x)$ for each scalar $\alpha$ and $x\in X$; 
		\item[(2)] $\rho(x+y)\leq \rho(x)+\rho(y)$ for $x,y\in X$.
\end{itemize}\end{definition} 
Suppose $X$ is a vector space  and $\tau$ is a group topology on $X$. We say that $(X,\tau)$ is an \textit{extended seminormed space} (esns, for short) if there exists a family $\mathcal{P}=\{\rho_i : i\in I \}$ of extended seminorms on $X$ such that the induced topology $\tau(\mathcal{P})$ (smallest topology under which each $\rho_i$ is continuous) on $X$ coincides with $\tau$. 

%

Extended locally convex spaces were introduced in \cite{esaetvs}. In order to define extended locally convex space, first we need to define fundamental extended locally convex space.	
\begin{definition}{\normalfont (\cite{esaetvs})}
	\normalfont		A  vector  space $X$ together with a group topology $\tau$ is said to be a \textit{fundamental extended locally convex space} $($fundamental elcs, for short$)$ if there exist two subspaces $M$ and $N$ such that 
	\begin{itemize}
		\item[(1)]  $X=M\oplus N$ and $X$ is topologically isomorphic to $(M, \tau|_M)\times (N, \tau|_N)$;
		\item[(2)] $(M, \tau|_M)$ is a locally convex space;
		\item[(3)] $(N,\tau|_N)$ is a discrete space.
	\end{itemize}
\end{definition}
\begin{definition}{\normalfont(\cite{esaetvs})}
	\normalfont	A vector space $X$ together with a group topology $\tau$ is said to be an \textit{extended locally convex space} (elcs, for short) if there is a collection of topologies (generating family) $\{\tau_\alpha : \alpha\in\Lambda\}$ on $X$ such that
	\begin{itemize}
		\item[(1)] for each $\alpha, (X, \tau_\alpha)$ is  a fundamental elcs;
		\item[(2)]  $\tau =\vee_{\alpha\in\Lambda}\tau_\alpha,$ that is, $\tau$ is the topology induced by the collection $\{\tau_\alpha:\alpha\in \Lambda\}.$
	\end{itemize}
\end{definition}
The \textit{ finite subspace} $X_{fin}$ 	for an  elcs  (or a fundamental elcs) $(X,\tau)$  is defined by $$X_{fin}=\bigcap \left\lbrace Y : Y ~\text{is an open subspace of}~ X\right\rbrace.$$
It is easy to see that $X_{fin}$ is the largest subspace of $X$ which is contained in every open subspace of $X$. The following facts for an elcs $(X,\tau)$ can be found in \cite{esaetvs}.
\begin{itemize}
	\item[(1)] The finite space $X_{fin}$ for $X$ is an open subspace if and only if $X$ is  a fundamental elcs (Corollary 3.11).
	\item[(2)] The subspace $(X_{fin}, \tau|_{X_{fin}})$ is a locally convex space, and it is the largest such subspace of $X$. 
	\item[(3)]  For every generating family $\{\tau_\alpha:\alpha\in\Lambda\}$ of $\tau$, we have  $$X_{fin}=\bigcap_{\alpha\in \Lambda} X_{fin}^\alpha, $$
\end{itemize} 
\noindent where $X_{fin}^\alpha$ is the finite subspace of $(X,\tau_\alpha)$ (Proposition 3.10).

In Theorem 4.3 of \cite{esaetvs}, it is shown that if $\tau$ is a topology on a vector space $X$, then $(X,\tau)$ is an esns if and only if it is an elcs. So every extended locally convex topology  is induced by a collection of extended seminorms. If $(X,\tau)$ is an elcs and the topology $\tau$ is induced by a collection $\{\rho_\alpha:\alpha\in \Lambda\}$ of extended seminorms on $X$, then $X_{fin}=\{x\in X:\rho_\alpha(x)<\infty~ \text{for all} ~\alpha\in \Lambda\}$. In fact, if $S(X,\tau)$ denotes the collection of all continuous extended seminorms on $(X,\tau)$, then $X_{fin} = \{x\in X:\rho(x)<\infty~ \text{for all} ~\rho\in S(X,\tau)\}$. By Proposition 4.7 of \cite{esaetvs} for a group topology $\tau$ on $X$, $(X,\tau)$ is an elcs if and only if there exists a neighborhood basis $\mathfrak{B}$ of $0_X$ consisting of absolutely convex (balanced and convex) sets. In the proof of Proposition 4.7 in \cite{esaetvs}, it is shown that if $\mathfrak{B}$ is a neighborhood base of $0_X$ in an elcs $(X,\tau)$ consisting of absolutely convex sets, then $\tau$ is induced by the collection $\{\rho_{V}:V\in \mathfrak{B}\}$ of Minkowski functionals.   

\begin{definition}\label{Minkowski funcitonal}		
	\normalfont	Suppose $(X, \tau)$ is an elcs and  $U$ is an absolutely convex subset of  $X$. Then the \textit{Minkowski functional} $\rho_U:X\rightarrow[0,\infty]$  for $U$  is defined as  $$\rho_U(x)=\inf\{\lambda>0: x\in \lambda U\}.$$
\end{definition}
\noindent Suppose $(X,\tau)$ is an elcs. Then the following facts are immediate from Definition \ref{Minkowski funcitonal}.
\begin{itemize}
	\item[(1)] Every Minkowski functional $\rho_{U}$ is an extended seminorm. In addition, if $U$ is absorbing, then $\rho_U$ is a seminorm on $X$.
	\item[(2)] $\rho_U$ is  continuous   if and only if   $U$ is a neighborhood of $0_X$. 
	\item[(3)]  If $U\subseteq V$, then $\rho_V\leq \rho_U.$ 
	\item[(4)]  In general, $\{x\in X:\rho_U(x)<1\}\subseteq U\subseteq \{x\in X:\rho_U(x)\leq 1\}.$ 
\end{itemize}  
We fix the following notation for the rest of this paper. Suppose $U$ is an absolutely convex subset of an elcs $X$. Then $X_{fin}^U$ is given by $\{x\in X: \rho_U(x)<\infty\}$. Similarly, if $\rho$ is a continuous extended seminorm on an elcs $X$, then $X_{fin}^\rho=\{x\in X: \rho(x)<\infty\}$. For any nonempty subset $A$ of a topological space $(X,\tau)$, we denote the closure and interior of $A$ in $(X,\tau)$ by $cl_\tau(A)$  and  $int_\tau(A)$, respectively. We refer to \cite{GTE, lcsosborne, tvsschaefer} for other terms and definitions.

	\section{finest locally convex topology of an elcs}
	The main aim of this section is to construct the finest locally convex topology coarser than the extended locally convex topology $\tau$ for an elcs $(X,\tau)$. We further show that the resulting locally convex space has the same dual as that of $(X,\tau)$. We also examine the separation of points and closed convex sets by a continuous linear functional.
	  


To formulate the finest locally convex topology for an elcs, we first give the following result,  a consequence of Proposition 4.7 proved in \cite{esaetvs}.
\begin{proposition}\label{base for elcs}
	Let $(X,\tau)$ be an elcs. Then there exists a neighborhood base $\mathfrak{B}$ of $0_X$ such that each element of $\mathfrak{B}$  is absolutely convex, and the collection $\{\rho_{V}:V\in \mathfrak{B}\}$ of Minkowski functionals generates $\tau$. 
\end{proposition}

	\begin{theorem}\label{construction}
		Suppose $(X,\tau)$ is an elcs. Then there exists a locally convex topology $\tau_{F}$  on $X$ with the following properties:
		\begin{itemize}
			\item[$(a)$] $\tau_{F} \subseteq \tau$;
			\item[$(b)$] if $\sigma $ is a locally convex topology on $X$ with $\sigma\subseteq \tau$, then $\sigma\subseteq \tau_{F}$.		
		\end{itemize} 
	\end{theorem}
	\begin{proof}
		By Proposition \ref{base for elcs}, $\tau$ is induced by the collection $\{\rho_V  : V\in \mathfrak{B} \}$ of Minkowski functionals, where  $\mathfrak{B}$ is a neighborhood base of $0_X$ in $(X,\tau)$ consisting of absolutely convex sets.  For $V\in\mathfrak{B}$, define
		$F_V= \{\rho : \rho$ \text{is a seminorm on} $X$, \text{and both } $\rho_V$ \text{and} $\rho$ induce same topology on ${X_{fin}^V}\}$, where ${X_{fin}^V} = \{x\in X: \rho_V(x) <\infty\}$. Let $F= \displaystyle{\cup_{V\in\mathfrak{B}} F_V}$. Suppose $\tau_{F}$ is the topology on $X$ induced by the collection $F$ of seminorms on $X$. Then $(X, \tau_{F})$ is a locally convex space.
		
		To show $\tau_F \subseteq \tau$, consider any $\rho\in F$. Then $\rho\in F_V$ for some $V\in \mathcal{B}$. Since $\rho$ and $\rho_V$ induce the same topology on $X_{fin}^V$ and $\rho_V(x)=\infty$ for $x\notin X_{fin}^V$, there exists a positive scalar $C_\rho$ such that $\rho\leq C_\rho  \rho_{V}$ on $X$. Consequently, the topology $\tau(\rho)$ induced by the seminorm $\rho$  is coarser than the topology $\tau(\rho_V)$ induced by the extended seminorm $\rho_V$. Consequently, $\tau(\rho)$ is coarser than $\tau.$ Since this holds for each $\rho\in F$, we have $\tau_F\subseteq \tau$.    
		
	    Let $\sigma$ be a locally convex topology on $X$ with $\sigma\subseteq \tau$, and let $\mathcal{U}$ be a neighborhood base of $0_X$ in $(X,\sigma)$ such that each element of $\mathcal{U}$ is absolutely convex and absorbing. To show $\sigma \subseteq \tau_F$, it is enough to show that each  $U \in \mathcal{U}$ is a neighborhood of $0_X$ in $(X,\tau_{F})$. Consider any $U\in \mathcal{U}$. Since $\sigma\subseteq \tau$, there exists $V\in \mathfrak{B}$ such that $V\subseteq U$. Suppose  $M_{V}$ is a linear subspace of $X$ such that $X=X_{fin}^{V}\oplus M_{V}.$ Define $\rho':X\rightarrow[0,\infty)$ by
		$$\rho' (x) = \rho_{V}(x_f)+\rho_{U}(x_M),$$ 
		where $x=x_f+x_M$ for $x_f\in X_{fin}^{V}$, $x_M \in M_{V}$, and $\rho_{V}$, $\rho_{U}$ are Minkowski functionals for $V$ and $U$, respectively. Then $\rho'$ is a seminorm on $X$ and $\rho' =\rho_{V}$ on $X_{fin}^V$. So  $\rho'\in F_V\subseteq F$. Consequently, $\{x\in X: \rho'(x)<1/2\} \in \tau_F$.
		
		We show that $\{x\in X: \rho'(x)<1/2\}\subseteq U$. Suppose  $\rho'(x=x_f+x_M)<\frac{1}{2}$. Then $\rho_{V}(x_f)<\frac{1}{2}$ and $\rho_{U}(x_M)<\frac{1}{2}$, which implies that $x_f\in \frac{V}{2}, x_M\in \frac{U}{2} $. Thus 
		\begin{align*}
		x=x_f+x_M\in \left( \frac{V}{2}+\frac{U}{2}\right)\subseteq \left( \frac{U}{2}+\frac{U}{2}\right)=U.	
		\end{align*} \end{proof}

	\begin{definition}
		\normalfont	The locally convex topology $\tau_{F}$ described in \textnormal{Theorem \ref{construction}} for an elcs  $(X,\tau)$ is called the \textit{finest locally convex topology coarser than the given topology $\tau$} (flc topology, for short).
	\end{definition}

Since every extended normed linear space $X$ is an elcs, the finest locally convex topology $\tau_{P_X}$ described in Theorem 2.7 of \cite{spoens} is equal to  $\tau_F$.

\begin{lemma}\label{condition for a cotinuous linear functional}
		Let $(X,\tau)$ be an elcs and let $\mathfrak{B}$ be a neighborhood base of $0_X$ consisting of absolutely convex sets. Then for every continuous seminorm $\rho$ on $X$, there exists a $W\in\mathfrak{B}$ such that $\rho\leq  \rho_W$. 
	\end{lemma}
	\begin{proof}
		Suppose $\rho$ is a continuous seminorm on $X$.  Then there exists a $W\in\mathfrak{B}$ such that $W\subseteq \rho^{-1}\left([0,1)\right)$. Note that if $x\in X$ and $\rho_W(x)<\alpha$, then $\frac{1}{\alpha}x\in W$. Consequently, $\rho(x)<\alpha$. Thus $\rho\leq \rho_W$.\end{proof} 
		
	
	\begin{theorem} \label{finest and original topology have same seminorm}
		Let $(X,\tau)$ be an elcs and let $\tau_F$ be its  flc topology. A seminorm on  $X$ is continuous with respect to $\tau$ if and only if it is continuous with respect to $\tau_{F}$.
	\end{theorem}
	\begin{proof}
		 As $\tau_F\subseteq \tau$, we only need to show that a continuous seminorm on $(X,\tau)$ is continuous with respect to $\tau_F$. Let $\mathfrak{B}$ be a neighborhood base of $0_X$ for $\tau$ such that members of $\mathfrak{B}$ are absolutely convex. Suppose $\rho:X\rightarrow [0,\infty) $ is a continuous seminorm on $(X, \tau)$. By Lemma \ref{condition for a cotinuous linear functional}, there exists a $W\in \mathfrak{B}$ such that $\rho\leq  \rho_W$. Let $X=X_{fin}^{W}\oplus M$. Define a seminorm $\rho':X\rightarrow[0,\infty)$  by  $$\rho'(x) = \rho_W(x_f)+\rho(x_M),$$ where $x=x_f+x_M$ for $x_f\in X_{fin}^{W}$ and $x_M\in M$. Since $\rho'=\rho_W$ on $X_{fin}^W$, $\rho'\in F_W\subseteq F$, where $F_W$ and $F$ are as defined in the proof of Theorem \ref{construction}.  Note that $\rho\leq \rho'$. Therefore $\rho$ is continuous with respect to $\tau_F$.
	\end{proof}
	\begin{corollary}\label{finest and original topology have same dual}
	Let $(X,\tau)$ be an elcs and let $\tau_F$ be its  flc topology. Then both $(X,\tau)$ and $(X,\tau_{F})$ have the same dual.	
	\end{corollary}
	\begin{proof}
		 Suppose $f\in (X,\tau)^*$. Then $|f|$ is a continuous seminorm on $(X,\tau)$. By Theorem \ref{finest and original topology have same seminorm}, $|f|$ is continuous with respect to $\tau_{F}$. Hence $f\in (X,\tau_F)^*$
	\end{proof}

\begin{definition}\label{finitely compatible norm} 
		\normalfont (\cite{spoens})	A \textit{finitely compatible norm} $\trinorm\cdot\trinorm$ for an extended normed linear space $(X,\parallel\cdot\parallel)$ is a usual norm on $X$ such that both $\trinorm\cdot\trinorm$ and $\parallel\cdot\parallel$ are equivalent on $X_{fin}$. 
		\end{definition}	
 The finest locally convex topology $\tau_{P_X}$ for an extended normed linear space  $(X,\parallel\cdot\parallel)$ is generated by the collection $\mathcal{P}$ of all finitely compatible norms on $X$ (Theorem 2.7, \cite{spoens}).
    
The following example shows that two different extended spaces may have the same flc topology.
	\begin{example}
		Let  $X=\mathbb{R}^2$. Then the extended norm  \[ \parallel(x, y)\parallel = \begin{cases}
		\text{$|y|,$} &\quad\text{if $x= 0 $}\\
		\text{$\infty$,} &\quad\text{if $x\neq 0,$ }\\
		\end{cases}\] and  the discrete extended norm  \[ \parallel(x,y)\parallel_{0,\infty}=
		\begin{cases}
		\text{$0,$} &\quad\text{if $x=y= 0 $}\\
		\text{$\infty$,} &\quad\text{otherwise }\\
		\end{cases}\] have different $X_{fin}$. Consequently, the extended normed spaces $(X, \parallel\cdot\parallel)$ and $(X, \parallel\cdot\parallel_{0,\infty})$ have different topologies. However, both spaces have the same finitely compatible norms as the dimension of $X$ is finite. Therefore the flc topologies of both spaces are equal.    
	\end{example}
	
	
Next we study separation of points from closed convex sets in an elcs $(X,\tau)$. But first we give the following result.  
	\begin{proposition} \label{Hausdorff condition for the flcs}
		Suppose $(X,\tau)$ is an elcs and $\tau_F$ is its  flc topology. Then $(X,\tau)$ is  Hausdorff if and only if $(X,\tau_{F})$ is Hausdorff.
	\end{proposition}
	\begin{proof}
		Suppose $(X,\tau)$ is Hausdorff. Let $x\in X$ and $x\ne 0_X$. So there exists a $\rho\in S(X,\tau)$ such that $\rho(x)\ne 0$. Let $M_\rho$ be a vector subspace of $X$ such that $X=X_{fin}^\rho\oplus M_\rho $, and let  $\{b_j:j\in J\}$ be a  basis for $M_{\rho}$. Define a seminorm $\rho':X\rightarrow [0,\infty)$ by
		$$\rho'(x)=\rho(x_F)+\sum_{j\in J}|\alpha_j|,$$ 
		where $x=x_F+x_M$ for $x_F\in X_{fin}^\rho$ and $x_M=\sum_{j\in J} \alpha_j b_j\in M$. Since $\rho\in S(X,\tau) $  and $\rho'\leq \rho$, $\rho'$ is  continuous with respect to $\tau$. By Theorem \ref{finest and original topology have same seminorm}, $\rho'$ is continuous with respect to $\tau_{F}$. If  $\rho'(x)= 0$, then $\rho(x_F)=0$ and $\sum_{j\in J} |\alpha_j|=0$. Thus $x_M=0_X$ and  $\rho(x)=\rho(x_F)=\rho'(x_F)=0$. We arrive at a contradiction. Therefore $\rho'(x)\neq 0$.  Hence $(X,\tau_{F})$ is Hausdorff.\end{proof}
	
The following corollary follows from Corollary \ref{finest and original topology have same dual} and Proposition \ref{Hausdorff condition for the flcs}.
\begin{corollary}\label{separation_of_points_by_linear_functional}
	Suppose $(X,\tau)$ is a Hausdorff elcs and $0_X \neq x \in X$. Then there exists $f \in X^*$ such that $f(x)\neq 0$.
\end{corollary}

The following example demonstrates that  $(X,\tau)$ and $(X,\tau_F)$ may have different closed convex sets.
	
	\begin{example}
		Suppose $X=\mathbb{R}^2$ and $\tau$ is the topology induced by the extended norm 	\[ \parallel(x,y)\parallel =
		\begin{cases}
		\text{$|x|,$} &\quad\text{if $y= 0 $}\\
		\text{$\infty$,} &\quad\text{if $y\neq 0.$ }\\
		\end{cases}\]
		Clearly, $X_{fin} = \{(x,0): x \in \mathbb{R}\}$. Let $A =\{ (x,y): x\in \mathbb{R},~ |y|< 1\}$. Then $A$ is absolutely convex and absorbing. Also $A = \bigcup_{|y|< 1}\{X_{fin}+(0,y)\}$ is a union of a locally finite family of closed sets. So $A$ is closed in $(X,\tau)$ (Corollary 1.1.12, \cite{GTE}).  If $A$ were closed in  $(X,\tau_{F})$, then $\{x\in X:\rho_A(x)\leq 1\}\subseteq A$ (Theorem 5.3.2, \cite{tvsnarici}), where $\rho_A$ is the Minkowski functional  for the set $A$. This is not possible as $\rho_A\left((0,1)\right) =1 $ and $(0,1)\notin A.$\end{example}

\begin{definition}
		Let $X$ be an  elcs  and $A\subseteq X$. Then $a\in core(A)$ if for each $x\in X$, there exists a $\delta_x>0$ such that $a+tx\in A$ for every $0\leq t\leq \delta_x.$  
	\end{definition}
The following facts are immediate from the above definition: $0_X \in core(A)$ implies $A$ is absorbent; $A \subseteq B$ implies $core(A) \subseteq core(B)$.  

The next result is analogous to Proposition 2.6 of \cite{spoens}.
	
\begin{proposition}\label{separation of a point and a closed set}
		Let $(X,\tau)$ be a real Hausdorff elcs with the flc topology $\tau_F$. Suppose $A\subseteq X$ is convex and  $b \notin A $. Then the following statements are equivalent.
		\begin{itemize}
			\item[(i)] There exists  a continuous linear functional $f$ with $f(b)<\inf_{x\in A} f(x)$. 
			\item[(ii)] There exists a convex set $B\subseteq X$ with $b\in core(B)\cap int_{\tau}(B)$ and $B\cap A=\emptyset$.
	\end{itemize}\end{proposition}
	\begin{proof}
		(i)$\Rightarrow$ (ii). Choose a scalar $\alpha $ with $f(b)<\alpha<\inf_{x\in A} f(x)$. Define $B=f^{-1}\left( (-\infty, \alpha] \right)$. Clearly, $B\cap A = \emptyset$ and $b \in int_{\tau}(B)$. To see $b \in core(B)$,  note that for any $x \in X$, we have $b+tx \in B$ for $0\leq t\leq \delta_x$, where  $\delta_x = \frac{\alpha - f(b)}{|f(x)|+1}$.
		
		 (ii)$\Rightarrow$ (i). Suppose $B$ satisfies (ii). Take $K=C\cap (-C)$ for $C= B-b$. Since  $b\in core(B)\cap int_{\tau}(B)$, we have $0_X\in core(K)\cap int_{\tau}(K)$. Hence $K$ is an absolutely convex and absorbing neighborhood of $0_X$ in $(X,\tau)$. Consequently, the Minkowski  functional $\rho_K$ is a  continuous seminorm  on $(X,\tau)$. By Theorem \ref{finest and original topology have same seminorm}, $\rho_K$ is continuous with respect to $\tau_F$. Since $b-a \notin K$ for any $a \in A$, we have $\rho_K(b, A) = \inf\{\rho_K(b-a): a \in A\}\geq 1$. So $b\notin cl_{\tau_F}{A}$. Hence by Theorem 4.25 of \cite{faaidg}, there is a continuous linear functional $f$ with  $f(b)<\inf_{x\in A} f(x).$  \end{proof}

	\begin{theorem}\label{closed convex sets in the flcs}
		Let $(X,\tau )$ be a real Hausdorff elcs with the flc topology $\tau_F$, and let $A\subseteq X$. If both $A$ and $A^c$ are convex, then the following statements are equivalent.
		\begin{itemize}
			\item[(i)] $core(A^c)= A^c$ and $A$ is closed with respect to  $\tau$.
			\item[(ii)]  $A$ is closed with respect to $\tau_F$.
		\end{itemize} 
	\end{theorem}
	\begin{proof}
		(i)$\Rightarrow$ (ii). Suppose  $x_0\notin A$. Then by the hypothesis  $x_0\in core(A^c)\cap int_{\tau}(A^c)$. By Proposition \ref{separation of a point and a closed set}, we can find a $\alpha\in\mathbb{R}$ and $f\in X^*$ with $x_0\in f^{-1}(-\infty, \alpha)$ and $A\cap  f^{-1}(-\infty, \alpha)=\emptyset$. Therefore $x_0\notin cl_{\tau_F}{A}.$
		
		(ii)$\Rightarrow$(i).  Let $x_0\notin A$. There exists a continuous linear functional $f$ such that $f(x_0)<\inf_{x\in A} f(x)$ (Theorem 4.25, \cite{faaidg}). By  Proposition \ref{separation of a point and a closed set}, there exists a $B\subseteq X$ such that $x_0\in core(B)$ and $B\cap A=\emptyset$. Thus $x_0\in core(A^c)$.\end{proof}

We end this section with an application of Theorem \ref{finest and original topology have same seminorm} in the classical function space theory. Recall that for a metric space $(X,d)$, $C(X)$ denotes the set of all continuous real valued functions on $(X,d)$. A \textit{bornology} on $X$ is a collection of nonempty subsets of $X$ that covers $X$ and is closed under finite union and taking subsets of its members. A \textit{base} for a bornology $\mathcal{B}$ is any subfamily $\mathcal{B}_0$ of $\mathcal{B}$ which is cofinal in $\mathcal{B}$ under set inclusion. If members of $\mathcal{B}_0$ are closed in $(X,d)$, then $\mathcal{B}$ is said to have a \textit{closed base}.
 
The most commonly used topology on $C(X)$ is the classical topology of uniform convergence on $\mathcal{B}$, usually denoted by $\tau_{\mathcal{B}}$ (see \cite{Suc} for definition). 

 
The topology $\tau_{\mathcal{B}}$ may be seen to be induced by the collection $\left\lbrace \rho_B:B\in \mathcal{B}_0 \right\rbrace$  of extended seminorms, where   $\rho_B(f) =\sup_{x\in B}|f(x)|$. So $(C(X), \tau_{\mathcal{B}})$ is an elcs.
 
For a bornology $\mathcal{B}$ on $(X,d)$ with a closed base, in \cite{Suc}, Beer and Levi introduced a variational form of $\tau_{\mathcal{B}}$ called  the topology of strong uniform convergence on $\mathcal{B}$, denoted by $\tau^{s}_{\mathcal{B}}$ (see \cite{Suc} for definition). 

 
The topology $\tau_\mathcal{B}^s$ on $C(X)$ can also be seen to be induced by extended seminorms of the form $$\rho_B^s(f) =\inf_{\delta>0}\left\lbrace \sup_{x\in B^\delta}|f(x)|\right\rbrace, $$ where $B\in \mathcal{B}_0$.    
 
In general, $\tau_{\mathcal{B}}$ and $\tau_{\mathcal{B}}^{s}$ are not equal on $C(X)$. However, if the bornology $\mathcal{B}$ is shielded from closed sets, then $\tau_\mathcal{B} =\tau_\mathcal{B}^s$ on $C(X)$ (Theorem 4.1 in \cite{Ucucas}).
 
For a nonempty $A\subseteq X$, a superset $A_1$ of $A$ is said to be a \textit{shield} for $A$ if for every closed subset $C$ of $X$ with $C\cap A_1 = \emptyset$, we have $C$ cannot be near to $A$, that is, there exists a $\delta > 0$ such that $C\cap A^\delta = \emptyset$ (\cite{Bcas}). Since $A^{\delta} = \overline{A}^{\delta}$, $A_1$ is a shield for $A$ if and only if it is a shield for $\overline{A}$. A bornology $\mathcal{B}$ is called \textit{shielded from closed sets} if for each $B\in \mathcal{B}$, $\mathcal{B}$  contains a shield for $B$. 	
 
 Our next result examines when the flc topologies for the extended locally convex spaces $(C(X), \tau_{\mathcal{B}})$ and $(C(X), \tau_{\mathcal{B}}^{s})$ coincide.  
 \begin{theorem}
 	Let $(X,d)$ be a metric space and let $\mathcal{B}$ be a bornology on $(X,d)$ with a closed base $\mathcal{B}_0$. Suppose $\tau_F$ and $\tau_{F}^s$ are the flc  topologies for $\left( C(X),\tau_\mathcal{B}\right) $ and $\left( C(X),\tau_\mathcal{B}^s\right)$, respectively. Then the following statements are equivalent.
 	\begin{itemize}
 		\item[(i)] $\tau_{F}=\tau_F^s$;
 		\item[(ii)] $\tau_\mathcal{B}=\tau_\mathcal{B}^s$;
 		\item[(iii)] $\mathcal{B}$ is shielded from closed sets.
 	\end{itemize}
 \end{theorem}
 \begin{proof}
 	The implication (iii)$\Rightarrow$ (ii)  follows from Theorem 4.1 of \cite{Ucucas}, and the implication (ii)$\Rightarrow$ (i) is immediate. 
 	
 	(i)$\Rightarrow$ (iii). Suppose there exists a closed set $B_0 \in \mathcal{B}$ such that $B_0$ has no shield in $\mathcal{B}$. So for every $B\in \mathcal{B}$ with $B_0\subseteq B$, there exists a closed set $C_B$ which is disjoint from $\overline{B}$ and $ (B_0)^{\delta} \cap C_B \neq \emptyset$ for each $\delta >0$. By Urysohn's lemma, for each $B \in \mathcal{B}$ with $B_0\subseteq B$, there exists a continuous function $f_B : (X,d) \to [0,1]$ such that $f_B(B)=\{0\}$ and $f_B(C_B)=\{1\}$. If we direct the set $\{B\in \mathcal{B},B_0\subseteq B\}$ by set inclusion, then the net $f_B\to 0$ in $(C(X),\tau_\mathcal{B})$. Consequently, $f_B\to 0$ in $(C(X),\tau_F)$. Let $C(X)_{fin}^{B_0}=\{h\in C(X): \rho_{B_0}^s(h)<\infty\}$. Suppose $M$ is a subspace of $C(X)$ with $C(X)=C(X)_{fin}^{B_0}\oplus M$. Define a seminorm $\rho:C(X)\rightarrow [0,\infty)$  by 
 	$$\rho(h)= \rho_{B_0}^s(h_f)+\mu(h_M), $$
 	where $h=h_f+h_M$ for $h_f\in C(X)_{fin}^{B_0}$, $h_M\in M$, and  $\mu $ is any seminorm on $M$. Since $\rho\leq \rho_{B_0}^{s}$,  $\rho$  is continuous with respect to $\tau_B^s$. By Theorem \ref{finest and original topology have same seminorm}, $\rho$ is continuous with respect to $\tau_F^s$.
 	
 	Observe that for each $B\in \mathcal{B}$ such that $B_0\subseteq B$, we have $f_B\in C(X)_{fin}^{B_0}$. Consequently, $\rho(f_B)= \rho_{B_0}^s(f_B)$. Since $ (B_0)^{\delta} \cap C_B \neq \emptyset$ for each $\delta >0$, it follows that $\rho(f_B)\geq 1$ for all $B \in \mathcal{B}$ with $B_0 \subseteq B$. Hence $f_B\nrightarrow 0$ in $(C(X), \tau_{F}^s)$. We arrive at a contradiction. \end{proof}

\section{finest locally convex topology on different spaces}
In this section, we find the flc topology for a subspace, quotient space, and also for a product of extended locally convex spaces. 

\begin{theorem}\label{finest topology for subspace}
	Let $(X,\tau)$ be an elcs with the flc topology $\tau_{F} $ and let $Y$ be a subspace of $X$. Then the topology $\tau_{F}|_Y$ on $Y$ is the the flc topology for the elcs $(Y,\tau|_Y)$. Moreover, $\left(Y, \tau_{F}|_Y \right)^*=\left( Y, \tau|_Y\right)^*.$
\end{theorem}

\begin{proof} As $\tau_{F}\subseteq \tau$, $(Y, \tau_{F}|_Y)$ is a locally convex space and $\tau_{F}|_Y\subseteq \tau|_Y$. To show $\tau_{F}|_Y$ is the flc topology for $(Y, \tau|_Y)$, let $\sigma $ be a locally  convex topology on $Y$ coarser than $\tau|_Y$. Let $0$ denote the common zero element of $X$ and $Y$. Suppose $U$ is an absolutely convex, absorbing neighborhood of $0$ in $(Y,\sigma)$. Then there exists an absolutely convex neighborhood $V$ of $0$ in $(Y,\tau|_Y)$ with $V\subseteq U$ and such that $V=W\cap Y$ for some  absolutely convex neighborhood $W$ of $0$ in $(X,\tau)$. If $\rho_{W}$ is the Minkowski functional for the set $W$ and $X_{fin}^W=\{x\in X:\rho_W(x)<\infty \}$,  then  $Y_{fin}^W= X_{fin}^W\cap Y$, where $Y_{fin}^W=\{y\in Y: \rho_W(y)<\infty\}$. Let $M_W$ be a linear subspace of $Y$ such that $Y=Y_{fin}^W\oplus M_W$. Since $M_W\cap X_{fin}^W=\{0\}$, there exists a subspace $N$ of $X$ such that $X=X_{fin}^W\oplus M_W\oplus N$. Let $P:X\rightarrow [0,\infty)$ be  a seminorm on $X$ defined by
	$$P(x)=\rho_{W}(x_f)+\rho_{U}(x_M)+\mu(x_N),$$
	where $x=x_f+x_M+x_N$ for  $x_f\in X_{fin}^W, x_M\in M_W, x_N\in N $, $\mu $ is  any seminorm on $N$, and $\rho_{U}$ is the Minkowski functional for $U$. Clearly, $P\leq \rho_{W}$. So $P$  is continuous with respect to $\tau$. Consequently, by Theorem \ref{finest and original topology have same seminorm}, $P$ continuous with respect to $\tau_{F}$. As $U\subseteq Y$ and $Y$ is a subspace, we have $$\left\lbrace y\in Y :P(y)<\frac{1}{2}\right\rbrace \subseteq \left(\frac{W}{2} +\frac{U}{2}\right)\bigcap Y\subseteq \left( \frac{V}{2}+\frac{U}{2}\right)  \subseteq U.$$
	Thus $U$ is a neighborhood of $0$ in $(Y,\tau_F|_Y).$ Hence $\sigma\subseteq \tau_{F}|_Y$. \end{proof}

\begin{corollary}{\normalfont(Continuous Extension theorem)} Let $(X,\tau)$ be an elcs with the flc topology $\tau_{F}$ and let $Y$ be a subspace of $X$. Then for any continuous linear functional $f$ on $Y$, there exists a continuous linear functional $\hat{f}$ on $X$ with $\hat{f}|_M=f$.  
\end{corollary}
\begin{proof}
	Let  $f\in (Y,\tau|_Y)^*$. By Theorem \ref{finest topology for subspace}, $f\in (Y,\tau_F|_Y)^*$. By Theorem 3.16 of \cite{lcsosborne}, there exists $\hat{f}\in (X,\tau_F)^* =(X,\tau)^*$ such that  $\hat{f}|_Y=f$.  \end{proof}

\begin{theorem}\label{finest topology for quotient space}
	Let $(X,\tau)$ be an elcs with the flc topology $\tau_F$ and let $Y$ be a closed subspace of $X$. Then the quotient topology $\pi\left(\tau_{F} \right) $ on $X/Y$ is the flc topology for the  elcs $\left( X/Y, \pi\left(\tau \right)\right)$, where $\pi(x) = x+Y$ for $x \in X$ denotes the quotient map. Moreover, $\left(X/Y,\pi(\tau) \right)^*=\left(X/Y,\pi(\tau_{F}) \right)^*.$  
\end{theorem}
\begin{proof}	As $\tau_F$ is a locally convex topology and $\tau_F\subseteq \tau$, the quotient space  $\left( X/Y,\pi(\tau_{F})\right)$ is a locally convex space and $\pi\left( \tau_{F}\right)\subseteq \pi\left(\tau \right)$. Suppose $\sigma$ is any locally convex topology on $X/Y$ coarser than $\pi\left(\tau \right)$. If $U$ is an absolutely convex, absorbing neighborhood of $0_{X/Y}$ in $\left(X/Y, \sigma \right)$, then there exists an absolutely convex neighborhood $V$ of $0_{X/Y}$ in  $\left(X/Y, \pi(\tau) \right)$ with $V\subseteq U$. Note that  $W=\pi^{-1}\left(V \right) $  is an absolutely convex neighborhood of $0_X$ in $\left( X,\tau\right)$. Let  $M_{W}$ be a linear subspace of $X$ such that $X=X_{fin}^{W}\oplus M_{W}$.  Define a seminorm $P:X\rightarrow [0,\infty)$ as follows
	$$P(x)=\rho_{W}(x_f)+\rho_{U}([x_M]),$$
	where $x=x_f+x_M$ for $x_f\in X_{fin}^{W}, x_M\in M_W$ and $[x_M]=x_M+Y$.  Then $P$ is continuous with respect to $\tau$ as $P\leq \rho_{W}$. By Theorem \ref{finest and original topology have same seminorm}, $P$ is continuous with respect to $\tau_F$.
	
	We show that  $\pi\left( \{x\in X:P(x)<\frac{1}{2} \}\right) \subseteq U $. If $P(x = x_f+x_M)<\frac{1}{2},$ then $\rho_{W}(x_f)<\frac{1}{2}$ and $\rho_{U}([x_M])<\frac{1}{2}$. Consequently,  $x_f\in \frac{W}{2}$ and $x_M+Y\in\frac{U}{2}$. Thus $\pi(x) = x_f+Y +x_M+Y\in \left(\frac{V}{2}+\frac{U}{2}  \right)\subseteq\left(\frac{U}{2}+\frac{U}{2}\right) \subseteq U.$ \end{proof} 	

\begin{theorem}\label{finest topology  for product space}
	Suppose $\{(X_i,\tau_i):i\in I \}$ is a family of extended locally convex spaces and for each $i\in I$,  $\tau_{F_i}$ is the corresponding flc topology. Then the product topology $\displaystyle{\Pi_{i\in I}} \tau_{F_i}$ is the  flc topology for the elcs  $(X,\tau)$, where $\tau$ is the product topology $\displaystyle{\Pi_{i\in I} \tau_{i}}$ on  $X=\displaystyle{\Pi_{i\in I} X_i}$. Moreover, $\left( X,\displaystyle{\Pi_{i\in I}\tau_{F_i}}\right)^*=\left(X,\tau \right)^*.  $   
	
\end{theorem}

\begin{proof} 
	Clearly, $\left(\displaystyle{\Pi_{i\in I} X_i}, \displaystyle{\Pi_{i\in I}\tau_{F_i}}\right) $ is a locally convex space and $\displaystyle{\Pi_{i\in I}\tau_{F_i}}\subseteq \tau$.  Let $\sigma$ be a locally convex topology on $X$ coarser than $\tau$ and let $U$ be an absolutely convex, absorbing neighborhood of $0_X$ in $(X,\sigma)$. We can find an absolutely convex neighborhood $V$ of $0_X$ in $(X,\tau)$ such that $V\subseteq U$. Without loss of generality, we can assume $V=\pi^{-1}_{i_1}\left( V_{i_1}\right) \cap \pi^{-1}_{i_2}\left(V_{i_2}\right)\cap\ldots\cap\pi^{-1}_{i_N}\left( V_{i_N}\right) $, where for each $1\leq j\leq N$,  $V_{i_j}$  is an absolutely convex neighborhood of $0_{X_{i_j}}$ in $(X_{i_j},\tau_{i_j})$. If $\rho_{V}$ is the Minkowski functional for $V$, then it is easy to see that $$X_{fin}^V= \pi^{-1}_{i_1}\left( X_{fin}^{V_{i_1}}\right) \cap \pi^{-1}_{i_2}\left(X_{fin}^{V_{i_2}}\right)\cap...\cap\pi^{-1}_{i_N}\left( X_{fin}^{V_{i_N}}\right),$$ where  for each $1\leq j\leq N$, $X_{fin}^{V_{i_j}}=\{z\in X_{i_j}: \rho_{V_{i_j}}(z)<\infty \}$.  For $1\leq j\leq N$,  suppose $M_{i_j}$ a linear subspace of $X_{i_j}$ such that $X_{i_j}=X_{fin}^{V_{i_j}}\oplus M_{i_j}$.  For each $j$, consider the seminorm $P_{i_j}:X_{i_j}\rightarrow [0,\infty)$ with 
	$$P_{i_j}(z)=\rho_{V_{i_j}}(z_f)+\rho_{U}(0,0,...,z_{M_{i_j}},0,0...),$$ 
	where $z=z_f+z_{M_{i_j}}$ for $z_f\in X_{fin}^{V_{i_j}},$ $z_{M_{i_j}}\in M_{i_j}$. Since  $P_{i_j}\leq \rho_{V_{i_j}}$,  each $P_{i_j}$ is  continuous with respect to  $ \tau_{i_j}$. By Theorem \ref{finest and original topology have same seminorm}, $P_{i_j}$ is continuous with respect to $\tau_{ F_{i_j}}$. So to show $\sigma \subseteq \displaystyle{\Pi_{i\in I}\tau_{F_i}}$,  it is enough to show that $$\bigcap_{j=1}^{N}\pi^{-1}_{i_j}\left( \left\lbrace z\in X_{i_j}: P_{i_j}(z)<\frac{1}{2N}\right\rbrace \right)\subseteq U.$$ 
	
	Let $T_{i_j}=\left\lbrace x= (x^i)\in \frac{U}{2N}: x^i=0~ \text{for} ~i\ne i_j \right\rbrace\subseteq X$ for $1\leq j\leq N$. Then  $$\bigcap
	_{j=1}^{N}\pi^{-1}_{i_j}\left( \left\lbrace z\in X_{i_j}: P_{i_j}(z)<\frac{1}{2N}    \right\rbrace \right) \subseteq \bigcap_{j=1}^{N}\left( \pi^{-1}_{i_j}\left( \frac{V_{i_j}}{2N} \right)  +  T_{i_j}\right)$$ and 	$$ \bigcap_{j=1}^{N} \pi^{-1}_{i_j}\left( \frac{V_{i_j}}{2N}\right)   + \sum_{j=1}^{N} T_{i_j} \subseteq \frac{V}{2N}+\frac{U}{2} \subseteq U.$$
	It remains to show that $$\bigcap_{j=1}^{N}\left( \pi^{-1}_{i_j}\left( \frac{V_{i_j}}{2N} \right)  +  T_{i_j}\right)\subseteq \bigcap_{j=1}^{N} \pi^{-1}_{i_j}\left( \frac{V_{i_j}}{2N}\right)   + \sum_{j=1}^{N} T_{i_j}.$$
	Suppose $y=(y^i)\in\bigcap_{j=1}^{N}\left( \pi^{-1}_{i_j}\left(\frac{V_{i_j}}{2N} \right)+T_{i_j}\right).$ Then for each $1\leq j\leq N$, there exist $v^{i_j}\in \frac{V_{i_j}}{2N}$ and $t^{i_j}\in X_{i_j}$ with $y^{i_j}=v^{i_j}+t^{i_j}$ and $(0,0,0,...,t^{i_j}, 0,0,...)\in T_{i_j}.$ Let $u=(u^i)$ with 	\[ u^{i}=
	\begin{cases}
	\text{$v^{i_j},$} &\quad\text{if $i=i_j,$ for $ j= 1, 2,...,N $}\\
	\text{$y^i$,} &\quad\text{otherwise, }\\
	\end{cases}\]
	and $w=(w^i)$ with \[ w^{i}=
	\begin{cases}
	\text{$t^{i_j},$} &\quad\text{if $i=i_j$, for $ j= 1, 2,...,N  $}\\
	\text{$0$,} &\quad\text{otherwise. }\\
	\end{cases}\]
	Then $u\in\bigcap_{j=1}^{N} \pi^{-1}_{i_j}\left( \frac{V_{i_j}}{2N}\right)  $ and  $w\in\sum_{j=1}^{N} T_{i_j}$ with $y=u+w$, which completes the proof.\end{proof}

 \section{Metrizability and normability of $\tau_F$}
 In this section,  we study conditions for an elcs $(X,\tau)$ to be extended normable and for the flc topology $\tau_{F}$ to be normable. We also examine the metrizability of $\tau_F$. Throughout this section, every elcs is assumed to be Hausdorff.
 
We first recall the definition of a bounded set in a locally convex space.
 
\begin{definition}\label{bouded sets_lcs} \normalfont	Let $A$ be a subset of a locally convex space $X$. Then $A$ is said to be bounded in $X$ if for every  neighborhood $U$ of $0_X$ there exists a scalar $c>0$ such that $A\subseteq c U.$ \end{definition}
 
Since in an elcs the neighborhoods of $0_X$ are not necessarily absorbing,  a finite subset of an elcs need not satisfy the above definition. For example, if $(X,\tau)$ is an elcs and $x\notin X_{fin}$, then we can find a neighborhood $U$ of $0_X$ such that $x \notin cU$ for any $c\neq 0$. This leads us to the next definition.
 
 \begin{definition}\label{bouded sets}\normalfont Let $(X, \tau)$  be an elcs and $A\subseteq X$. Then $A$ is said to be bounded in $(X, \tau)$ if for every neighborhood $U$ of $0_X$ there exist $c>0$ and a finite subset $P\subseteq A$ such that $A\subseteq P+c U.$ \end{definition}
 
 \begin{remark}	If $(X,\tau)$ is a locally convex space, then Definitions \ref{bouded sets_lcs} and \ref{bouded sets} are equivalent.  
 \end{remark}     
 
 The following facts about bounded subsets of an elcs $(X,\tau)$ can be proved easily using  Definition \ref{bouded sets}.
 \begin{itemize}
 	\item[(a)] Every finite subset of $X$ is bounded.
 	\item[(b)] Every subset of a bounded set is bounded.
 	\item[(c)] Suppose $(X,\tau)$, $(Y,\sigma)$ are two elcs and $T:X\rightarrow Y$ is a  continuous linear operator. If $A$ is bounded in $(X,\tau)$, then $T(A)$ is bounded in $(Y,\sigma)$.
 	\item[(e)] No subspace $\left( \text{other than} \{0_X \}\right)$  is bounded.
 	\item[(f)] Suppose $(x_n)$ is a convergent sequence in an elcs $(X,\tau)$. Then $\{x_n:n\in \mathbb{N}\}$ is  bounded in  $(X, \tau)$. 
 	\item[(g)] If $A$ is bounded in $(X,\tau)$, then $A$ is also bounded in $(X,\tau_{F})$.
 \end{itemize}
 
 \begin{remark}
 	Suppose $A$  is a bounded set in an elcs $\left(X,\tau \right) $ and  $\tau$ is  induced by a collection $\{\rho_\alpha: \alpha\in \Lambda\}$ of extended seminorms. Then for each $\alpha\in \Lambda $, there exists a finite set $P_\alpha \subseteq A$ such that $A\subseteq P_\alpha+ X_{fin}^\alpha$. In general, we can not replace $X_{fin}^\alpha$ by $X_{fin}$.\end{remark}
 
 \begin{example}
 	Let $X=C_{00}= \text{span} \{e_n:n\in\mathbb{N}\}$, where $e_n=(0,...,1,0,...)$. For each $n\in\mathbb{N}$, define 
 	$\rho_n:X\rightarrow [0,\infty]$ by \[\rho_n\left((x^i) \right) =
 	\begin{cases}
 	\text{$\infty,$} &\quad\text{if $x^n\ne 0$}\\
 	\text{$\sup_{i\in\mathbb{N}}|x^i|$,} &\quad\text{if $x^n= 0$. }\\ 
 	\end{cases} \] 
 	If $\tau$ is the topology induced by the collection $\{\rho_n: n\in\mathbb{N}\}$ of extended seminorms, then $\left(X,\tau \right)$ is an elcs and $X_{fin}=\{0_X\}$. Let $A=\{e_n:n\in\mathbb{N}\}$, then $A$ is bounded in $\left( X,\tau\right) $  and  $A\nsubseteq P+X_{fin}$ for any finite set $P.$ \end{example}
 
 \begin{proposition}\label{extended norm by minkowski functional}  Let $(X,\tau)$ be an elcs. Then the following statements hold.
 	\begin{itemize}
 		\item[(i)]  If $A$ is an absolutely convex and bounded subset of $X$, then $A\subseteq X_{fin}$. 
 		\item[(ii)]  If $U$ is an absolutely convex and bounded neighborhood of $0_X$ in $(X,\tau)$, then the Minkowski functional $\rho_U$ is a continuous extended norm on $X$, and the finite subspace $X_{fin}$ of $(X,\tau)$ is equal  to $\{x\in X:\rho_{U}(x)<\infty\}$.
 \end{itemize}\end{proposition}
 
 \begin{proof}(i). Observe that for any $x \in A$, the line segment $\{tx:t\in[0,1]\}$ joining the point $x$ and $0_X$ is contained in $A$. If $x\in A\setminus X_{fin}$, that is,  $\rho(x) = \infty$ for some $\rho \in S(X,\tau)$, then for $U = \rho^{-1}[0,1)$ and any $c >0 $, there does not exist a finite subset $P$ of $\{tx:t\in[0,1]\}$ such that $\{tx:t\in[0,1]\} \subseteq P+ cU$. Consequently, $\{tx:t\in[0,1]\}$ is an unbounded subset of $A$. We arrive at a contradiction. 
 	
 (ii). Suppose $U$ satisfies the given hypothesis. We show that the Minkowski functional $\rho_{U}$ is a continuous extended norm on $(X,\tau)$. We only need to show if  $\rho_{U}(x)=0$, then $x=0_X$. By contradiction, suppose $\rho_{U}(x)=0$ but $x\neq 0_X$. By Corollary \ref{separation_of_points_by_linear_functional}, there exists $f\in X^*$ such that $f(x)\neq 0$. Since $\rho_{U}(x)=0$, $\{nx:n\in\mathbb{N}\}\subseteq U$. Therefore $f(U)$ is not bounded. Which contradicts that $U$ is bounded.  
 	
As $U$ is absolutely convex and bounded, by (i),  $U\subseteq X_{fin}$. Consequently, $\{x\in X:\rho_{U}(x)<\infty\}\subseteq X_{fin}$. Since $X_{fin}$ is contained in every open subspace of $(X,\tau)$, we have $X_{fin}=\{x\in X:\rho_{U}(x)<\infty\}.$  \end{proof}
 
 \begin{lemma}\label{equivalence of two elcs} Let $\tau_1$ and $\tau_2$ be two extended locally convex topologies on a vector space $X$ and let $M$ be an open subspace of $(X,\tau_1)$ and $(X,\tau_2)$. If $\tau_1$ and $\tau_2$ are equal on $M$, then $\tau_1$ and $\tau_2$ are equal on $X$. 
 \end{lemma}
 
 \begin{theorem} \label{Extended normability of elcs}
 	Suppose an  elcs $(X, \tau)$ has a bounded neighborhood of $0_X$. Then there exists an extended norm $\parallel\cdot\parallel$  on $X$ with $ X_{fin}=\{x\in X:\parallel x\parallel<\infty\}$ and the topology  $\tau_{\parallel\cdot\parallel}$ induced by the extended norm  $\parallel\cdot\parallel$   is equal to $\tau$.
 \end{theorem}
 
 \begin{proof}
 	Suppose $U$ is a bounded neighborhood of $0_X$ in $(X,\tau)$. Then there exists a bounded absolutely convex neighborhood $V$ of $0_X$ such that $ V\subseteq U$. Consequently, by Proposition \ref{extended norm by minkowski functional}, $V\subseteq X_{fin}$ and  the Minkowski functional $\rho_V$ is an extended norm on $X$ with $X_{fin}=\{x\in X: \rho_{V}(x)<\infty \}$. Define  $\parallel x\parallel =\rho_V(x)$ for all $x \in X$. Since $V$ is a bounded neighborhood of $0_X$ in the locally convex space $(X_{fin},\tau|_{X_{fin}})$, the topology induced by $\parallel\cdot\parallel$ coincides with  $\tau|_{X_{fin}}$ on $X_{fin}$ (see, Theorem 6.2.1 of \cite{tvsnarici}). As  $X_{fin}$ is an open subspace of both $(X,\tau)$ and $(X,\parallel\cdot\parallel)$, by Lemma \ref{equivalence of two elcs}, $\tau_{\parallel\cdot\parallel}$ is equal to $\tau$.\end{proof}        
 
 
We now study the metrizability of $\tau_{F}$.  Recall that an elcs $(X,\tau)$ is metrizable if and only if $(X,\tau)$ is a countable elcs, that is, it has a countable neighborhood base of $0_X$. However, the metrizability of $(X,\tau)$ does not imply the metrizability of the space $(X,\tau_F)$ (see, Example \ref{tau_is_metrizable_but_tau_F_is_not_metrizable}). 

\begin{proposition}{\normalfont(Theorem 5.6.2, \cite{tvsnarici})} \label{discontinuous linear functional on infinite dimensional metrizable lcs}If $X$ is an infinite dimensional metrizable locally convex space, then there exists a discontinuous  
 	linear functional on $X$. 
\end{proposition}
 
 \begin{theorem}\label{metrizability of the flcs}
 	Suppose $(X,\tau)$ is a countable elcs. Then the following statements are equivalent.
 	\begin{itemize}
 		\item[(a)] The space $(X,\tau_F)$ is metrizable.
 		\item[(b)] The dimension of the quotient space $X/ X_{fin}^\rho$ is finite for each $\rho\in S(X,\tau)$.
 \end{itemize}  \end{theorem}	
 
 \begin{proof}(a)$\Rightarrow$(b). Suppose $\tau_{F}$ is metrizable and  $\dim\left(X/X_{fin}^\rho \right) $ is not finite for some $\rho\in S(X,\tau)$. So there exists an infinite dimensional subspace $M$ of $X$ such that $X = X_{fin}^{\rho} \oplus M$. By Proposition \ref{discontinuous linear functional on infinite dimensional metrizable lcs}, there exists a discontinuous linear functional $\phi$ on $(M,\tau_{F}|_M)$. If we extend $\phi$ linearly on $X$ so that $\phi\left( X_{fin}^\rho\right)=0$, then $\phi$ cannot be continuous on $(X,\tau_{F})$. Since $\phi\left( X_{fin}^\rho \right)=0 $, we have $\phi\in (X,\tau)^* = (X,\tau_F)^*$.  We arrive at a contradiction.
 	
 	(b)$\Rightarrow$(a). Suppose $\mathfrak{B}=\{U_n:n\in \mathbb{N}\}$ is a  countable neighborhood base  of $0_X$ for $(X,\tau)$ such that each element of $\mathfrak{B}$ is absolutely convex.  For each $n\in\mathbb{N}$, there exists a finite dimensional subspace $M_n$ such that $X=X_{fin}^{U_n}\oplus M_n$, where $X_{fin}^{U_n} = \{x\in X: \rho_{U_n}(x) < \infty\}$. Let $\{z_j:j\in\Delta_n\}$ be a basis for $M_n$ and let $\rho'_n:X\rightarrow[0,\infty)$ be a seminorm defined by $$\rho'_n\left(x \right) =\rho_{U_n}(x_f)+\sum_{j\in\Delta_n}|\alpha_j|,$$
 	where $x=x_f+\sum_{j\in\Delta_n}\alpha_jz_j$. Since $\rho'_n\leq \rho_{U_n}$, $\rho_n'$ is a continuous seminorm on $(X,\tau)$. By Theorem \ref{finest and original topology have same seminorm},  $\rho'_n$ is continuous on $(X,\tau_F)$.
 	
 	If $\sigma$ is the locally convex topology on $X$ induced by the collection $\mathcal{P'}=\{\rho_n':n\in\mathbb{N}\}$, then $\sigma\subseteq \tau_F\subseteq \tau$. If $\rho$ is a continuous seminorm on $(X,\tau_F)$, then by Lemma \ref{condition for a cotinuous linear functional}, there exists an $n_0 \in \mathbb{N}$ such that $\rho\leq \rho_{U_{n_0}}$. Thus $\rho\leq \rho'_{n_0}$. Hence $\rho$ is continuous with respect to $\sigma$. \end{proof}
 
\begin{corollary}\label{metrizability of flc in enls}
Suppose $X$ is an enls. Then the following statements are equivalent.
\begin{itemize}
	\item[(a)] The finest locally convex space $(X,\tau_F)$ is metrizable.
	\item[(b)] $X$ is almost conventional, that is, the dimension of $X/X_{fin}$ is finite. 
\end{itemize}	
\end{corollary} 
 
 In general, we can neither replace $X_{fin}^\rho$ with $X_{fin}$ nor can drop the condition $\dim\left( X/X_{fin}^\rho\right) $ is finite for each $\rho\in S(X,\tau)$ in Theorem \ref{metrizability of the flcs}. 
 
 \begin{example}
 	Suppose $X=C_{0}$ is the collection of all real sequences which converge to $0$. For each $n\in\mathbb{N}$, define $\rho_n:X\rightarrow [0,\infty]$ by
 	\[\rho_n((x^i))=
 	\begin{cases}
 	\text{$\infty$,} &\quad\text{ if $x^i\neq 0$  for some $1\leq i\leq n$}\\
 	\text{$\sup|x^i|$,} &\quad\text{if $x^i= 0$ for  $1\leq i\leq n$.}\\
 	\end{cases}
 	\]
 	Suppose $\tau$ is the topology induced by the collection $\mathcal{P}=\{\rho_n:n\in\mathbb{N} \}$ of extended norms. Then $(X,\tau)$ is an elcs and $X_{fin}=\{0_X\}$. We show that the supremum  norm topology $\tau_{\parallel\cdot\parallel_\infty}$ is equal to $\tau_{F}$. Since $\parallel\cdot\parallel_\infty\leq \rho_1$, we  have  $\tau_{\parallel\cdot\parallel_\infty}\subseteq \tau$. By Theorem \ref{construction}, $\tau_{\parallel\cdot\parallel_\infty}\subseteq \tau_{F}.$ 
 	
 	Conversely, let $\rho$ be a continuous seminorm on $(X,\tau_{F})$. Since $\rho_n\leq \rho_{n+1}$ for each $n$, by Proposition 4.4 of \cite{esaetvs}, there exist  $n_0\in\mathbb{N}$  and $c>1$ such that $\rho\leq c\rho_{n_0}$. Let $X=X_{fin}^{\rho_{n_0}}\oplus M$. Consider the norm $\trinorm\cdot\trinorm:X\to [0,\infty)$ defined by $$\trinorm x\trinorm=\parallel x_f\parallel_\infty + \parallel x_M\parallel_\infty+ \rho(x_M),$$ where $x=x_f+x_M$ for $x_f\in X_{fin}^{\rho_{n_0}}$ and $x_M\in M.$ Then both $\trinorm\cdot\trinorm$ and $\parallel\cdot\parallel_\infty$ are finitely compatible norms for $(X, \rho_{n_0})$ as $\parallel \cdot\parallel_\infty=\trinorm \cdot\trinorm= \rho_{n_0}$ on $X_{fin}^{\rho_{n_0}}$. Since $(X, \rho_{n_0})$ is an extended Banach space and the dimension of the quotient space $ X/X_{fin}^{\rho_{n_0}}$ is finite, by Corollary 2.8 of \cite{spoens}, all finitely compatible norms for $(X,\rho_{n_0})$ are equivalent. Therefore $\trinorm\cdot\trinorm$ is equivalent to $\parallel\cdot\parallel_{\infty}$. It is easy to see that $\rho\leq c\trinorm\cdot \trinorm$ on $X$. Consequently, $\rho$ is continuous with respect to $\trinorm\cdot\trinorm$, so continuous with respect to $\parallel\cdot\parallel_\infty.$ Hence $\tau_{F}\subseteq \tau_{\parallel\cdot\parallel_\infty}.$ \end{example}
 
  \begin{example}\label{tau_is_metrizable_but_tau_F_is_not_metrizable}
 	Suppose $X$ is infinite dimensional and $\parallel\cdot\parallel_{0,\infty}$ is the discrete extended norm on $X$. Then $(X,\parallel\cdot\parallel_{0,\infty})$ is an extended normed linear space, and $\dim (X/X_{fin})$ is infinite. So by Corollary \ref{metrizability of flc in enls}, the flc topology $\tau_F$ for $(X,\parallel\cdot\parallel_{0,\infty})$ is not metrizable. 
 \end{example}

 Finally, we study when $(X,\tau_F)$ is normable and give a condition on $(X,\tau)$ for which $(X,\tau_F)$ is barreled.
  
 \begin{definition}\label{Mackey topology}
 	\normalfont   	Suppose $(X,\tau)$ is a locally convex space. Then the Mackey topology $\tau_M$ on $X$ is defined by the family 
 	$$\mathcal{B}_M=\left\lbrace A_\circ: A~ \text{is weak}^* \text{compact subset of}~ X^* \text{and} \left(A_\circ \right)^\circ=A \right\rbrace, $$  
 	where  $A_\circ=\{x\in X: |f(x)|\leq 1 ~\forall f\in A \}$, and $(A_\circ)^\circ=\{f\in X^*: |f(x)|\leq 1 \forall ~ x\in A_\circ\}.$   
 \end{definition}
 
 If the topology $\tau=\tau_M$, then the locally convex space $(X,\tau)$ is  called a \textit{Mackey space}. For more details, see \cite{tvsnarici, lcsosborne, mmitvs}. 
 
 We recall the following results related to Mackey topology from \cite{tvsnarici}.
  
 \begin{proposition}\label{Mackey topology results}
 	Let $(X,\tau)$ be a locally convex space. Then the following statements hold. 
 	\begin{itemize}
 		\item[(i)] The Mackey topology  is the largest locally convex topology on $X$, preserving all continuous linear functionals of $X$ {\normalfont(Theorem 8.7.4, \cite{tvsnarici})}.
 		\item[(ii)] If $(X,\tau)$ is  a metrizable locally convex space, then the given topology $\tau$ is equal to its Mackey topology $\tau_M$ {\normalfont(Example 8.8.10, \cite{tvsnarici})}.
 		\item[(iii)] Both $(X,\tau)$ and $(X,\tau_M)$ have the same  bounded sets {\normalfont(Theorem 8.8.7, \cite{tvsnarici})} .
 	\end{itemize}
 \end{proposition}
 
 \begin{theorem}\label{normability of the flcs}
 	Suppose $\left(X,\tau \right) $ is a countable elcs and $\parallel\cdot\parallel$ is a norm on $X$  such that $\left(X,\parallel\cdot\parallel \right)^*=\left( X,\tau\right)^*$. Then $\tau_{F}$ is induced by the norm $\parallel\cdot\parallel$.	
 \end{theorem}
 
 \begin{proof}
 	Suppose $\parallel\cdot\parallel$ is a norm on $X$ with the given property. Then by Corollary \ref{finest and original topology have same dual}, $(X,\parallel\cdot\parallel)^* =( X,\tau_F)^*$. By (i) and (ii) of Proposition \ref{Mackey topology results}, both $(X,\parallel\cdot\parallel)$ and $(X,\tau_F)$ have the same Mackey topology $\tau_M$, and the norm topology $\tau_{\parallel\cdot\parallel}$ coincides with $\tau_M$. Hence $\tau_F\subseteq \tau_{\parallel\cdot\parallel}$. 
 	
 	For reverse inclusion, let  $\{U_n:n\in\mathbb{N}\}$ be a countable neighborhood base of $0_X$ in $\left( X,\tau\right)$ such that $U_{n+1}\subseteq U_n$ for each $n\in\mathbb{N}$. We show that the norm topology is coarser than $\tau$. Let $B_X$ be the closed unit ball in $\left( X,\parallel\cdot\parallel\right)$. If $B_X$  is not a neighborhood of $0_X$ in $(X,\tau)$, then  $nB_{X}$ is also not a neighborhood of $0_X$ in $(X,\tau)$ for any $n\in\mathbb{N}$. So  there exists $x_n\in U_n$ such that $x_n\notin nB_{X}$ for every $n\in\mathbb{N}$. It is easy to see that $x_n\to 0$ in $\left( X,\tau\right)$. Hence $\{x_n:n\in\mathbb{N}\}$ is bounded in $\left( X,\tau\right)$. Consequently, $\{x_n:n\in\mathbb{N}\}$ is bounded in $\left( X,\tau_F\right)$. By Proposition \ref{Mackey topology results}(iii), $\{x_n:n\in\mathbb{N}\}$ is bounded in $\left( X,\parallel\cdot\parallel\right)$. We arrive at a  contradiction as $\{x_n:n\in\mathbb{N}\}\nsubseteq nB_X$ for any $n\in\mathbb{N}.$ Thus by Theorem \ref{construction},  $\tau_{\parallel\cdot\parallel}\subseteq \tau_F$.   \end{proof}

 \begin{definition}
 	A locally convex space $(X,\tau)$ is said to be  barreled if every closed, absolutely convex, and absorbing subset $U$ of $X$ is a neighborhood of $0_X$ in $(X,\tau).$  
 \end{definition}
 
 \begin{theorem}\label{barrelledness}
 	Suppose $(X,\parallel\cdot\parallel)$ is an extended Banach space. Then the corresponding finest locally convex space $(X,\tau_{F})$ is Barreled.
 \end{theorem}
 
 \begin{proof}
 	Let $U$ be an absolutely convex, absorbing, and closed set in $(X,\tau_{F})$. Then $nU$ is closed in $(X, \parallel\cdot \parallel)$ for each $n\in \mathbb{N}$, and  $X=\bigcup_{n\in\mathbb{N}}nU$. Since $(X, \parallel\cdot \parallel)$ is complete, by Baire Category theorem, the interior of $n_0U$ in $(X,\parallel\cdot\parallel)$ is nonempty for some $n_0 \in \mathbb{N}$. So there exist $x\in n_0U$ and $r>0$ such that the open ball $B(x,r)\subseteq n_0U$. As $U$ is absolutely convex, we have $B\left(0_X, \frac{r}{2n_0}\right)=\frac{-x}{2n_0}+B\left(\frac{x}{2n_0}, \frac{r}{2n_0}\right) \subseteq U$. Hence $U$ is a neighborhood of $0_X$ in $(X,\parallel\cdot \parallel)$. Consequently, the Minkowski functional $\rho_U$  is a continuous seminorm on $(X, \parallel\cdot \parallel)$. By Theorem \ref{finest and original topology have same seminorm}, $\rho_U$ is a continuous seminorm on $(X,\tau_{F})$. Since $\rho_U^{-1}\left( [0,1)\right)\subseteq U$, $U$ is a neighborhood of $0_X$ in $(X,\tau_{F})$.\end{proof}
 
 \begin{corollary}\label{mackeyness of the flcs}
 	Suppose $\left( X,\parallel \cdot\parallel\right)$ is an extended Banach space. Then  $(X,\tau_{F})$ is a Mackey space.
 \end{corollary}
 \begin{proof}	It follows as every barreled space is a Mackey space (Corollary 4.9, \cite{lcsosborne}).  \end{proof}

\bibliographystyle{plain}
\bibliography{reference_file}

\end{document}